\documentclass[11pt,letterpaper]{amsart}
\usepackage{amscd}
\usepackage{amssymb}
\usepackage{amsthm}
\usepackage{enumerate}
\usepackage[normalem]{ulem}
\usepackage{mathtools}
\usepackage[usenames,dvipsnames]{xcolor}
\usepackage{stmaryrd}
\usepackage[left=1in,top=1in,right=1in]{geometry}
\usepackage[T2A]{fontenc}
\usepackage[utf8]{inputenc}
\usepackage{mathrsfs}

\newtheorem{thm}{Theorem}[section]
\newtheorem{lemm}[thm]{Lemma}
\newtheorem{prop}[thm]{Proposition}
\newtheorem{cor}[thm]{Corollary}

\theoremstyle{definition}
\newtheorem{defn}[thm]{Definition}

\numberwithin{equation}{section}

\newcommand{\bC}{{\mathbb C}}

\newcommand{\bN}{{\mathbb N}}

\newcommand{\bR}{{\mathbb R}}

\newcommand{\cA}{{\mathcal A}}

\newcommand{\cE}{{\mathcal E}}

\newcommand{\cH}{{\mathcal H}}

\newcommand{\cL}{{\mathcal L}}
\newcommand{\cM}{{\mathcal M}}

\newcommand{\cR}{{\mathcal R}}
\newcommand{\cS}{{\mathcal S}}

\newcommand{\actson}
{\curvearrowright}

\DeclareMathOperator{\Proj}{Proj}
\DeclareMathOperator{\Prob}{Prob}

\DeclareMathOperator{\Fix}{Fix}
\DeclareMathOperator{\Bern}{Bern}

\DeclareMathOperator{\id}{id}
\DeclareMathOperator{\Hamm}{Hamm}

\DeclareMathOperator{\tr}{tr}

\DeclareMathOperator{\Sub}{Sub}
\DeclareMathOperator{\Span}{span}

\DeclareMathOperator{\dom}{dom}
\DeclareMathOperator{\ran}{ran}

\DeclareMathOperator{\IRS}{IRS}

\newcommand{\ip}[1]{\langle #1 \rangle}

\newcommand\Sym{\operatorname{Sym}}
\newcommand\Stab{\operatorname{Stab}}

\begin{document}



\title{Approximate homomorphisms and sofic approximations of orbit equivalence relations}

\author{Ben Hayes}
\address{\parbox{\linewidth}{Department of Mathematics, University of Virginia, \\
141 Cabell Drive, Kerchof Hall,
P.O. Box 400137
Charlottesville, VA 22904}}
\email{brh5c@virginia.edu}
\urladdr{https://sites.google.com/site/benhayeshomepage/home}

\author{Srivatsav Kunnawalkam Elayavalli}
\address{\parbox{\linewidth}{Department of Mathematics, University of California, San Diego, 9500 Gilman Drive \# 0112, La Jolla, CA 92093}}
\email{srivatsav.kunnawalkam.elayavalli@vanderbilt.edu}
\urladdr{https://sites.google.com/view/srivatsavke}

\begin{abstract}

We show that for every countable group, any sequence of approximate homomorphisms with values in permutations can be realized as the restriction of a sofic approximation of an orbit equivalence relation. Moreover, this orbit equivalence relation is uniquely determined by the invariant random subgroup of the approximate homomorphisms. We record applications of this result to recover various known stability and conjugacy characterizations for almost homomorphisms of amenable groups. 
\end{abstract}


\maketitle

\section{Introduction}

This paper is a follow-up to \cite{GenEZ}. We are once again interested in approximate homomorphisms of groups. Recall that  approximate homomorphisms are sequences of maps $\sigma_{n}\colon G\to H_{n}$ where $G,H_{n}$ are groups, each $H_{n}$ has a bi-invariant metric $d_{n}$ and the sequence satisfies
\[\lim_{n\to\infty}d_{n}(\sigma_{n}(gh),\sigma_{n}(g)\sigma_{n}(h))=0  \textnormal{ for all $g,h\in G$}.\]
In our previous article, we studied automorphic conjugacy of sofic approximations. 
In this article, we are interested in when approximation homomorphisms with values in permutation groups are asymptotically conjugate via a sequence of permutations.

One motivation for the study of asymptotic conjugacy of approximate homomorphisms of groups is the notion of \emph{stability}. Stability of homomorphisms dates back to the 40's via work of Hyers \cite{UlamStable} (answering a question of Ulam), and asks when approximate homomorphisms can be perturbed to a sequence of honest homomorphisms. E.g. we say a group is \emph{permutation stable}  if every sequence of approximate homomorphisms of the group with values in permutation groups is pointwise close to a sequence of honest homomorphisms.
See \cite{AP1, APStable, BLT, EckShul, ESS, HS2, IoanaStabilityT}, for work in this direction. The notion of stability gives a potential approach to proving the existence of a nonsofic group  \cite{glebskyrivera,AP1,BBFlexStab}, has connections to cohomology \cite{CGLT, DQuasiRep, DMatrixStability}, and to operator algebras \cite{CDAlmostFlat, DAsyStable, HT, ScottSri2019ultraproduct}.

While not immediate, there is a natural connection between stability and asymptotic conjugacy. For example, one could consider a modification of permutation stability where we only demand that every \emph{sofic approximation} (instead of every approximate homomorphism) is close to an honest homomorphism (this concept is introduced in \cite{AP1} under the name weak stability). When the group is amenable, then as shown in \cite[Theorem 1.1]{AP1}, the  Kerr-Li \cite[Lemma 4.5]{KLi2} and Elek-Szabo \cite{elekxzabo} uniqueness theorems imply that this modified version of permutation stability is equivalent to residual finiteness.

The situation for permutation stability of general asymptotic homomorphisms of amenable groups was fully classified by Becker-Lubotzky-Thom \cite[Theorem 1.3]{BLT} who showed that an amenable group $G$ is permutation stable if and only if every invariant random subgroup  of $G$  is a limit of IRS's coming from actions on finite sets. Invariant random subgroups arose from the works \cite{MVY16, BergeronGaboriau, StuckZimmer, Vershik12}, we recall the definition in \ref{defn: IRS}. A different proof of this was given in \cite[Theorem 3.12]{APStable}, which reformulates the results in terms of traces.

Given a sequence $(\sigma_{n})_{n}$ of approximate homomorphisms and a free ultrafilter $\omega$, one can  naturally produce an invariant random subgroup  associated to $(\sigma_{n})_{n},\omega$ which we denote by $\IRS(\sigma_{\omega})$ (see Definition \ref{defn: IRS of approx homom}).
Given an invariant random subgroup $\Theta$ of $G$, following \cite[Proposition 13]{MVY16}, we can  construct a generalized Bernoulli shift action $G\actson (X_{\Theta},\mu_{\Theta})$ associated to $\Theta$. Let $\cR_{G,X_{\Theta}}$ be the orbit equivalence relation of this action. We show that for every sequence of asymptotic homomorphisms of $G$ whose given IRS is $\Theta$, we may extend the asymptotic homomorphism to a sofic approximation (in the sense of \cite{ElekLip}) of $\cR_{G,X_{\Theta}}$. This extension result holds for any countable group $G$.

\begin{thm}\label{thm: main theorem intro}
Let $G$ be a group, and let $(\sigma_{n})_{n=1}^{\infty}$ be a sequence of approximate homomorphisms of $G$ and $\omega$ a free ultrafilter on $\bN$. Let $\sigma_{\omega}$ be the ultraproduct of this sequence. Let $\Theta$ be the IRS of $(\sigma_{\omega})$, and let $\cR_{G,X_{\Theta}}$ be the orbit equivalence relation of the $\Theta$-Bernoulli shift over $G$ with base $[0,1]$. Then $\sigma_{\omega}$ extends to a sofic approximation of $\cR_{G,X_{\Theta}}$.
\end{thm}

In the case where $\Theta$ is $\delta_{\{1\}}$, our result says that any sofic approximation of $G$ extends to a sofic approximation of any Bernoulli shift over $G$. This was previously proved in \cite{ElekLip, PaunSofic,PoppArg, Bow} (in fact, \cite{PoppArg} proves that the generalized Bernoulli shift $\cR_{G\actson (X,\mu)^{G/H}}$ is sofic if $G$ is sofic, $H$ is amenable and $(X,\mu)$ is any probability space).  The proof of the case $\Theta=\delta_{\{1\}}$ given in \cite[Proposition 7.1]{ElekLip} (see also \cite[Theorem 3.1]{CMPSpecial} and \cite[Theorem 8.1]{Bow}) provided inspiration for the proof we give of Theorem \ref{thm: main theorem intro}.

In the context of Theorem \ref{thm: main theorem intro}, if $G$ is amenable, then $\cR_{G,X_{\Theta}}$ is hyperfinite by Ornstein-Weiss \cite{OWAnnounce}, \cite[II \S 3]{OrnWeiss} (see Connes-Feldman-Weiss \cite{CFW} for more general results), and thus any two sofic approximations of $\cR_{G,X_{\Theta}}$ are conjugate \cite[Proposition 1.20]{PaunSofic}.
Using this conjugacy fact, the results of \cite[Theorem 3.12]{APStable},\cite[Theorem 1.3]{BLT} are corollaries of Theorem \ref{thm: main theorem intro}.
This again illustrates the utility of approximate conjugacy results in the context of stability. 

\subsection*{Comments on some applications} To demonstrate the utility of Theorem \ref{thm: main theorem intro} we recover with separate proof ideas, some results characterizing conjugacy for almost representations of amenable group appearing in \cite{NSStable, ELEK20122593, AP1, BLT}. Namely, those results are proved via usage of the asymptotic combinatorial notation of hyperfinite graphs, and on Benjamini-Schramm convergence. Our proof is in some sense more ``continuous", and is based purely in ergodic theory and probabilistic arguments. In particular, we do not need the combinatorial notion of hyperfiniteness of graph sequences or Benjamini-Schramm convergence. 
In some sense, one can view ergodic theory as a limit of combinatorics  and so our methods can be viewed as a limiting version of those in \cite{NSStable, ELEK20122593, APStable,BLT}. Both the combinatorial and continuous approaches have their utility. We believe one benefit of our approach is that it reveals that one can work directly with the limiting object, and does not always have to resort to working with combinatorics at the finitary level and taking limits. 
\subsection*{Organization of the paper}  We begin in Section \ref{sec:preliminaries} by recalling some background on metric groups and approximate homomorphisms. We also give the definition of the IRS of a sequence of approximate homomorphisms here and restate Theorem \ref{thm: main theorem intro} in these terms. In Section \ref{sec: OE background}, we recall the background on orbit equivalence relations and their sofic approximations we need and state Theorem \ref{thm: main theorem intro} in these terms. In Section \ref{sec: applications} 
we deduce \cite[Theorem 3.12]{APStable} from \ref{thm: main theorem intro}, and also deduce  \cite[Theorem 1.3]{BLT} from Theorem \ref{thm: main theorem intro}. We also explain the connection between action traces and IRS's in this section. In Section \ref{sec: real Bern proof}, we prove Theorem \ref{thm: main theorem intro}.


\section{Preliminaries}\label{sec:preliminaries}

Throughout we consider $G$ to be a countable group. Let $\Sym(n)$ denote the finite symmetric group of rank $n$. The normalized  Hamming distance, which is a bi-invariant metric on $\Sym(n)$, is given by $$d_{\Hamm}(\chi_{1}, \chi_{2})= \frac{|\{i: \chi_{1}(i)\neq \chi_{2}(i)\}|}{n}.$$ Recall the following.

\begin{defn}\label{defn: IRS}
    A sequence of maps $\sigma_{n}\colon G\to \Sym(d_{n})$ are said to be approximate homomorphisms if for all $g,h\in G$ we have $$\lim_{n\to \infty} d_{\Hamm}(\sigma_n(gh), \sigma_n(g)\sigma_n(h))=0.$$
\end{defn}

Let $\omega$ be a free ultrafilter on $\mathbb{N}$. Let $(G_n,d_n)$ be countable groups with bounded bi-invariant metrics. Denote by $$\prod_{n\to \omega} (G_n,d_n)= \{(g_n)_{n\in \mathbb{N}}\}/\{(g_n): \lim_{n\to \omega} d_n(g_n,1_{G_n})=0\}. $$ Observe that by the bi-invariance property of the metrics $d_n$ the subgroup   $\{(g_n): \lim_{n\to \omega} d_n(g_n,1_n)=0\}$ is a normal subgroup. If $(g_{n})_{n}\in \prod_{n}G_{n},$ we let $(g_{n})_{n\to\omega}$ denote its image in the ultraproduct. 
In the above context, a sequence of approximate homomorphisms $\sigma_{n}\colon G\to \Sym(d_{n})$ naturally produces a homomorphism into $\prod_{n\to\omega}(\Sym(d_{n}),d_{\Hamm})$ by
\[\sigma_{\omega}(g)=(\sigma_{n}(g))_{n\to\omega}.\]
For a nonnegative integer $k$, we use $[k]=\{1,\cdots,k\}$. We also set $[0]=\varnothing$.

Suppose that $X$ is a compact, metrizable space. Assume $(Y,\nu)$ is a standard probability space and that we have a Borel map $Y\to \Prob(X)$ given by $y\mapsto \mu_{y}$. Then by the Riesz representation theorem, we can define $\int_{Y}\mu_{y}\,d\nu(y)$ to be the unique probability measure $\eta$ satisfying
\[\int_{X} f\,d\eta=\int_{Y}\int_{X}f\,d\mu_{y}\,d\nu(y)\]
for all $f\in C(X)$.


\subsection{Preliminaries on IRS's}\label{IRS prelims}

Given a sequence of approximate homomorphisms $\sigma_{n}\colon G\to \Sym(d_{n})$, define $S_{\sigma_{n}}\colon [d_{n}]\to \{0,1\}^{G}$ by $(S_{\sigma_{n}})(j)(g)=1_{\{j\}}(\sigma_{n}(g)(j))$. Throughout the paper, for a finite set $E$, we use $u_{E}$ for the uniform measure on $E$ and we typically use $u_{d}$ instead of $u_{[d]}$. Set $\Theta_{n}=(S_{\sigma_{n}})_{*}(u_{d_{n}})$. It turns out that subsequential limits of $\Theta_{n}$ can be nicely described in terms of well known objects.

\begin{defn}\label{defn: basics of IRS}
Let $G$ be  a countable, discrete group. We let $\Sub(G)$ be the set of subgroups of $G$, which we regard as a subspace of $\{0,1\}^{G}$ by identifying a subgroup with its indicator function. We equip $\Sub(G)$ with the topology induced from this inclusion. Then $\Sub(G)$ is a closed subset of $\{0,1\}^{G}$ and is thus compact in the induced topology. We let $\IRS(G)$ be the set of probability measures on $\Sub(G)$ which are invariant under the conjugation action of $G$ given by $g\cdot H=gHg^{-1}$ for all $g\in G, H\in \Sub(G)$.
\end{defn}
We equip $\Prob(\{0,1\}^{G})$ with the weak$^{*}$-topology (viewing complex Borel measures on $\{0,1\}^{G}$ as the dual of $C(\{0,1\}^{G})$).
We often regard $\Prob(\Sub(G))$ as a closed subset of $\Prob(\{0,1\}^{G})$ by identifying $\Prob(\Sub(G))$ with the probability measures which assign $\Sub(G)$ total mass one.
In \cite[Proposition 13]{MVY16}, it is shown that $\Theta\in \IRS(G)$ if and only if there is a   probability measure-preserving action $G\actson (X,\mu)$ so that $\Theta=\Stab_{*}(\mu)$ where $\Stab\colon X\to \Sub(G)$ is given by $\Stab(x)=\{g\in G:gx=x\}$. So IRS's are a naturally occurring construction when considering general (i.e. not assumed essentially free) probability measure-preserving actions. One can think of a sofic approximation as a sequence of almost free almost actions on finite sets. From this perspective, it is reasonable to expect IRS's to arise when one considers more general almost actions (i.e. asymptotic homomorphisms) that are not asymptotically free almost actions. 
The following lemma explains exactly how IRS's arise from approximate homomorphisms.

\begin{lemm}
Let $G$ be  a countable, discrete group and $\sigma_{n}\colon G\to \Sym(d_{n})$ approximate homomorphisms.
 For every free ultrafilter $\omega\in \beta\bN\setminus\bN$ we have that
\[\Theta_{\omega}=\lim_{n\to\omega}\Theta_{n}\in \IRS(G).\]
\end{lemm}

\begin{proof}
Let $G\actson \{0,1\}^{G}$ be given by $(g\cdot x)(h)=x(g^{-1}hg)$. We use $\alpha$ for the induced action on functions: $(\alpha_{g}f)(x)=f(g^{-1}\cdot x)$ for all $f\colon \{0,1\}^{G}\to \bC$ Borel and all $g\in G$. Note that if $F\subset G$ is finite, then we can find a sequence $\Omega_{n}\subseteq [d_{n}]$ (depending upon $F$) with $u_{d_{n}}(\Omega_{n})\to 1$ so that for all $j\in \Omega_{n}$,  all $1\leq k\leq (2023)!$, all $g_{1},\cdots,g_{k}\in F\cup F^{-1}\cup \{e\}$, and all $s_{1},\cdots,s_{k}\in \{\pm 1\}$ we have
\[\sigma_{n}(g_{1}^{s_{1}}\cdots g_{k}^{s_{k}})(j)=\sigma_{n}(g_{1})^{s_{1}}\cdots\sigma_{n}(g_{k})^{s_{k}}(j).\]
Then if $g,h\in F$ and $j\in \Omega_{n}$, we have:
\[(g\cdot S_{n}(j))(h)
=S_{n}(\sigma_{n}(g)j)(h).\]
This shows that:
\begin{equation}\label{eqn:almost IRS invariance}
\lim_{n\to\infty}u_{d_{n}}\left(\bigcap_{h\in E}\{j:(g\cdot S_{n}(j))(h)=S_{n}(\sigma_{n}(g)j)(h)\}\right)=1 \textnormal{ for all $g\in G,E\subseteq G$ finite.}
\end{equation}
For a finite $E\subseteq G$, let $A_{E}=\{f\circ \pi_{E}:f\in C(\{0,1\}^{E})\}$ where $\pi_{E}\colon \{0,1\}^{G}\to \{0,1\}^{E}$ is given by $\pi_{E}(x)=x\big|_{E}$. It follows from (\ref{eqn:almost IRS invariance}) and permutation invariance of $u_{d_{n}}$ that:
\[\lim_{n\to\infty}\int \alpha_{g}(f)\,d\Theta_{n}-\int f\,d\Theta_{n}=0 \textnormal{ for all $g\in G$, $f\in \bigcup_{E\subseteq G \textnormal{finite}}A_{E}$}.\]
Thus
\[\int f\,d\Theta_{\omega}=\int \alpha_{g}(f)\,d\Theta_{\omega}\]
for all $g\in G,f\in \bigcup_{E}A_{E}$. Stone-Weierstrass implies that $\bigcup_{E}A_{E}$ is norm dense in $C(\{0,1\}^{G})$ and so by the Riesz representation theorem we have shown that $\Theta_{\omega}$ is invariant under the conjugation action of $G$.

It thus suffices to show that $\Theta_{\omega}$ is supported on the space of subgroups of $G$. For $g,h\in G$ let:
\[X_{g,h}=\{x\in \{0,1\}^{G}:x(gh)\geq x(g)x(h)\},\]
\[I_{g}=\{x\in \{0,1\}^{G}:x(g^{-1})=x(g)\}.\]
Then
\[\Sub(G)=\bigcap_{g\in G}I_{g}\cap \bigcap_{g,h\in G}X_{g,h}\cap \{x\in \{0,1\}^{G}:x(e)=1\}.\]
Since $G$ is countable, it suffices to show that $\Theta_{\omega}$ assigns each set in this intersection measure $1$. By the Portmanteau theorem \cite[Theorem 11.1.1]{DudleyProb}, for each $g,h\in G$:
\begin{align*}
&\Theta_{\omega}(\{x\in \{0,1\}^{G}:x(gh)\geq x(g)x(h)\})\\
&\geq\lim_{n\to\omega}u_{d_{n}}(\{j:1_{\{j\}}(\sigma_{n}(gh)(j))\geq 1_{\{j\}}(\sigma_{n}(g)(j))1_{\{j\}}(\sigma_{n}(h)(j))\})\\
&\geq \lim_{n\to\omega}u_{d_{n}}(\{j:\sigma_{n}(gh)(j)=\sigma_{n}(g)\sigma_{n}(h)(j)\})=1.
\end{align*}
 The proofs that $\Theta_{\omega}$ assigns measure $1$ to $I_{g}$ and $\{x\in \{0,1\}^{G}:x(e)=1\}$ are similar.

\end{proof}

We are thus able to make the following definition.

\begin{defn}\label{defn: IRS of approx homom}
Let $G$ be a countable, discrete group and $\sigma_{n}\colon G\to \Sym(d_{n})$ approximate homomorphisms. Define $\Theta_{n}$ as before Definition \ref{defn: basics of IRS}. Given $\omega\in \beta\bN\setminus\bN$, we define $\IRS(\sigma_{\omega})=\lim_{n\to\omega}\Theta_{n}$.
\end{defn}
We remark on an alternate construction of $\IRS(\sigma_{\omega})$. One can take an ultraproduct of the measure spaces $(\{1,\cdots,d_{n}\},u_{d_{n}})$ to obtain a probability space $(\cL,u_{\omega})$ called the Loeb measure space \cite{Loebmeasure}.
The actions $\Sym(d_{n})\actson \{1,\cdots,d_{n}\}$ along with the approximate homomorphisms $\sigma_{n}$ will induce a probability measure-preserving action on $(\cL,u_{\omega})$ in a natural way. Under this action, one can show that $\IRS(\sigma_{\omega})$ is $\Stab_{*}(u_{\omega})$ where $\Stab\colon \cL\to \Sub(G)$ is given by $\Stab(z)=\{g\in G:gz=z\}$. However, we will not need this fact and thus will not prove it.
If $\IRS(\sigma_{\omega})$ does not depend upon $\omega,$ then $\lim_{n\to\infty}\Theta_{n}$ exists. In this case, we call $\lim_{n\to\infty}\Theta_{n}$ the \emph{stabilizer type of $\sigma_{n}$}. E.g.  $\lim_{n\to\infty}\Theta_{n}=\delta_{\{1\}}$ if and only if $\sigma_{n}$ is a sofic approximation.



\section{Background on Orbit equivalence relations and  Theorem \ref{thm: main theorem intro}}\label{sec: OE background}

In order to extend approximate homomorphisms to sofic approximations of relations, we need the following construction (appearing first in \cite{MVY16}) of an action  associated to an IRS, say $\Theta$. The intention is that the action is ``Bernoulli as possible" while still having $\Theta$ as its IRS. For technical reasons the action will not always have $\Theta$ as its IRS, but under mild conditions (which we state precisely after the definition) it will, and we think it is still worth stating the general construction.

\begin{defn}\label{defn: real Bernoulli shift}
Let $G$ be a countable, discrete group and let $\Theta\in \IRS(G)$. Let $X$ be a compact metrizable space and $\nu$ a Borel probability measure on $G$. Let
\[Y=\{(H,x)\in \Sub(G)\times X^{G}:x(hg)=x(g) \textnormal{ for all $g\in G,h\in H$}\}.\]
Observe that $Y$ is a closed subset of $\Sub(G)\times
X^{G}$. For $H\in \Sub(G)$, let $G/H$ be the space of right cosets of $H$ in $G$ and regard $\nu^{\otimes G/H}$ as a probability measure on $X^{G}$ which is supported on the $x\in X^{G}$ which are constant on right $H$-cosets.
Let $\mu_{\Theta}$ be the measure
\[\mu_{\Theta}=\int_{\Sub(G)}\delta_{H}\otimes \nu^{\otimes G/H}\,d\Theta(H).\]
 We sometimes denote $(Y,\mu_{\Theta})$ as $\Bern(X,\nu,\Theta)$.
 We let $G\actson X^{G}$ by
 \[(gx)(a)=x(g^{-1}a) \textnormal{ for all $g,a\in G,x\in X^{G}$.}\]
 Note that $G\actson Y$ by 
 \[g\cdot(H,x)=(gHg^{-1},gx).\]
 We call $G\actson (Y,\mu_{\Theta})$ the \emph{$\Theta$-Bernoulli action with base $(X,\nu)$.}
\end{defn}

Suppose that $\nu$ is not a dirac mass, and choose $b\in [0,1)$ so that $\nu(\{x\})\in [0,b]$ for $x\in X$.
Suppose that $H\in \Sub(G)$ and $F\subseteq G/H$ is finite with $H\in F$, then by Fubini-Tonelli
\begin{align*}
    \nu^{\otimes G/H}(\{x\in X^{G}:x(g_{1})&=x(g_{2})\textnormal{ for all $c_{1},c_{2}\in G/H$, and all $g_{1}\in c_{1},g_{2}\in c_{2}$}\})\\
    &\leq \nu^{\otimes G/H}(\{x\in X^{G}:x(g)=x(e)\textnormal{ for all $c\in F$, and all $g\in c$}\})\\
    &\leq b^{|F|}.
\end{align*}
Since $b<1$, and the above inequality is true for every finite $F\subseteq G$, we deduce that if either
\begin{itemize}
    \item $\nu$ is atomless, or
    \item $H$ is infinite index in $G$ a.s.,
\end{itemize}
then for $\mu_{\Theta}$-a.e. every $(H,x)$ with $H\ne G$ we have $\Stab((H,x))=H$. Note that $\Stab((G,x))$=G. So, under either of the above bulleted conditions, the action
$G\actson (Y,\mu_{\Theta})$ has IRS equal to $\Theta.$ 
This construction first appears in \cite[Proposition 13]{AGYIRS} (see \cite[Section 3.2-3.3]{CPIRS} for the locally compact case, as well as a proof of the fact that if $\Theta$ is ergodic then we can modify the above construction to get an ergodic action). See \cite{SewardSinai}, \cite[Section 5]{RobinTDWeakEquiv} for further applications of this construction.

The classical Bernoulli shift with base $(X,\nu)$ is the case when $\Theta=\delta_{\{1\}}$. It is of great importance in ergodic theory, via its connections to probability (it is the sample space for i.i.d. $X$-valued random variables $(\Upsilon_{g})_{g\in G}$). It also has many desirable properties such as being mixing (\cite[Section 2.5]{PetersenET} and \cite[Section 2.3]{KerrLiBook}), complete positive entropy \cite{RudolphMixing, KerrCPE}, Koopman representation being an infinite direct sum of the left regular \cite[Section 2.3]{KerrLiBook}, and being a free action when the acting group is infinite (proved above). It is also canonically associated to any group. We refer the reader to \cite[Section 4.9]{Walters}, \cite[Section 6.4 and 6.5]{PetersenET}, and \cite[Section 2.3]{KerrLiBook} for more details and information on the classical Bernoulli shift.
For our purposes, we will only need that the $\Theta$-Bernoulli shift retains a residue of freeness, in that under the above conditions it has $\Theta$ as its IRS.

We will need the notion of an orbit equivalence relation. 

\begin{defn}
A \emph{discrete, probability measuring preserving equivalence relation} is a tuple $(X,\nu,\cR)$ where $(X,\nu)$ is a standard probability space, $\cR\subseteq X\times X$ is Borel, and so that the following holds:
\begin{itemize}
    \item (equivalence relation) the relation $x\thicksim y$ given by $(x,y)\in \cR$ is an equivalence relation,
    \item (discreteness) for almost every $x\in X$, we have
    \[[x]_{\cR}=\{y\in X:(x,y)\in \cR\}\]
    is countable,
    \item (pmp) for every Borel $f\colon \cR\to [0,+\infty]$ we have 
    \[\int \sum_{y\in [x]_{\cR}}f(x,y)\,d\nu(x)=\int \sum_{y\in [x]_{\cR}}f(y,x)\,d\nu(x).\]
\end{itemize}
\end{defn}

The last item can be recast as follows: define a Borel measure $\overline{\nu}$ on $\cR$ by
\[\overline{\nu}(B)=\int_{X}|\{y\in [x]_{\cR}:(x,y)\in B\}|\,d\nu(x).\]
Then the map $(x,y)\mapsto (y,x)$ preserves the measure if and only if the last item holds. This implies, for example, that if  $\cR$ is a discrete, probability measure preserving equivalence relation on $(X,\nu)$ and if $f\in L^{1}(\cR)$, then
\[\int \sum_{y\in [x]_{\cR}}f(x,y)\,d\nu(x)=\int \sum_{y\in [x]_{\cR}}f(y,x)\,d\nu(x).\]

If $G$ is  a countable discrete group and $G\actson (X,\nu)$ is a probability measure preserving action, we then have a discrete, probability-measure relation given as the \emph{orbit equivalence relation}
\[\cR_{G,X}=\{(x,gx):x\in X,g\in G\}.\]
All discrete, probability measure-preserving equivalence relations arise this way \cite{FelMoore}.

As mentioned before, we will extend approximate homomorphisms to sofic approximations of equivalence relations. In order to define a sofic approximation of an equivalence relation, we use tracial von Neumann algebras.

\begin{defn}
Let $\cH$ be a Hilbert space. A unital $*$-subalgebra $M$ of $\mathbb{B}(\mathcal{H})$ is said to be a von Neumann algebra if it is closed in the weak operator topology given by the convergence  $T_n\to T$ if $\langle (T_n-T)v, w\rangle\to 0$ for all $v,w\in \mathcal{H}$. 
A \emph{projection} in $M$ is an element $p\in M$ with $p=p^{*}=p^{2}$. We let $\Proj(M)$ be the set of projections in $M$.
A normal \emph{homomorphism} between von Neumann algebras $M,N$ is a linear $\pi\colon M\to N$ which preserves products and adjoints and such that $\pi\big|_{\{x\in M:\|x\|\leq 1\}}$ is weak operator topology continuous. Such maps are automatically norm continuous \cite[Proposition 1.7 (e)]{ConwayOT}.
We say that $\pi$ is an \emph{isomorphism} if is bijective, it is then automatic that $\pi^{-1}$ is a normal homomorphism \cite[Proposition 46.6]{ConwayOT}.
A pair $(M,\tau)$ is a \emph{tracial von Neumann algebra} if $M$ is a von Neumann algebra and $\tau$ is a trace, meaning that $\tau: M\to \bC$  satisfies:
\begin{itemize}
    \item $\tau$ is linear,
    \item $\tau(x^{*}x)\geq 0$ and  $\tau(x^{*}x)>0$ if $x\ne 0$,
    \item $\tau(ba)=\tau(ab)$ for all $a,b\in M$,
    \item $\tau(1)=1$,
    \item $\tau\big|_{\{x\in M:\|x\|\leq 1\}}$ is weak operator topology continuous.
\end{itemize}
Given a Hilbert space $\cH$ and $E\subseteq B(\cH)$, we let $W^{*}(E)$ be the von Neumann algebra generated by $E$. 
\end{defn}
For a tracial von Neumann algebra $(M,\tau)$ and $x\in M$, we set $\|x\|_{2}=\tau(x^{*}x)^{1/2}$. 
A simple example of a tracial von Neumann algebra is the following. For $k\in \bN$, define $\tr\colon M_{k}(\bC)\to \bC$ by
\[\tr(\upsilon)=\frac{1}{k}\sum_{j=1}^{k}\upsilon_{jj}.\]
Then $(M_{k}(\bC),\tr)$ is a tracial von Neumann algebra.
The following is folklore, but we highlight it because we will use it explicitly. This can be proved, e.g. by following the discussion in Section 2 of \cite{BekkaOAsuperrigid}.
\begin{lemm} \label{lem: folklore ish}
Let $(M_{j},\tau_{j}),$ $j=1,2$ be tracial von Neumann algebras. Suppose that $N\subseteq M_{1}$, is a  weak operator topology dense $*$-subalgebra, and that $\pi\colon N\to M_{2}$ is a $*$-homomorphism with $\tau_{2}\circ \pi=\tau_{1}\big|_{N}$. Then $\pi$ extends uniquely to a trace-preserving, normal $*$-homomorphism from $M_{1}\to M_{2}$.

\end{lemm}

Given a discrete, probability measure-preserving equivalence relation $\cR$ on $(X,\nu)$, we let $[\cR]$ be the group of all bimeasurable bijections $\gamma\colon X_{0}\to Y_{0}$ where $X_{0},Y_{0}$ are conull subsets of $X$ and with $\gamma(x)\in [x]_{\cR}$ for almost every $x\in X$. As usual, we identify two such maps if they agree off a set of measure zero. The group $[\cR]$ is called the full group of $\cR$. We let $[[\cR]]$ be the set of all bimeasurable bijections $\gamma\colon B_{1}\to B_{2}$ (where $B_{1},B_{2}$ are measurable subsets of $X$) which satisfy that $\gamma(x)\in [x]_{\cR}$ for almost every $x\in B_{1}$. As usual, we identify two such maps if they agree off a set of measure zero. We usually use $\dom(\gamma),\ran(\gamma)$ for $B_{1},B_{2}$ above. We define maps $\vartheta\colon L^{\infty}(X,\nu)\to B(L^{2}(\cR,\overline{\nu}))$ and $\lambda\colon [[\cR]]\to \mathcal{U}(L^{2}(\cR,\overline{\nu}))$ by
\[(\vartheta(f)\xi)(x,y)=f(x)\xi(x,y),\]
\[(\lambda(\gamma)\xi)(x,y)=1_{\ran(\gamma)}(x)\xi(\gamma^{-1}(x),y).\]
We define the \emph{von Neumann algebra of the equivalence relation} to be
\[L(\cR)=W^{*}(\lambda([\cR])\cup \vartheta(L^{\infty}(X,\nu))).\]
Note that if $B\subseteq X$ is measurable, then we have an element $\id_{B}\in [[\cR]]$ with $\dom(\id_{B})=E=\ran(\id_{B})$ and $\id_{B}(x)=x$ for all $x\in B$. Moreover, $\lambda(\id_{B})=\vartheta(1_{B})$. Since simple functions are dense in $L^{\infty}(X,\nu)$, this implies that
\[L(\cR)=W^{*}(\lambda([[\cR]])).\]
The von Neumann algebra $L(\cR)$ is equipped with a trace
\[\tau(x)=\ip{x1_{\Lambda},1_{\Lambda}}\]
where $\Lambda=\{(x,x):x\in X\}$. We typically identify $L^{\infty}(X,\nu)$ and $[[\cR]]$ as subsets of $L(\cR)$ and do not make explicit reference to the maps $\lambda,\vartheta.$

Another example of a tracial von Neumann algebra is the ultraproduct of tracial von Neumann algebras. Let $\omega$ be a free ultrafilter on $\mathbb{N}$. Suppose $(N_k, \tau_k)$are tracial von Neumann algebras.  Denote the ultraproduct by $$\prod_{k\to \omega} (N_k,\tau_k)= \{(x_k)_{k\in \mathbb{N}}\ |\ \sup_{k} \|x_k\|< \infty\}/\{(x_k)| \lim_{k\to \omega} \|x_k\|_2=0\}. $$
If $(x_{k})_{k}\in \prod_{k}N_{k}$ with
 $\sup_{k}\|x_{k}\|<+\infty,$ we use $(x_{k})_{k\to\omega}$ for its image in $\prod_{k\to\omega}(N_{k},\tau_{k})$.
By the proof of \cite[Lemma A.9]{BrownOzawa2008}
the ultraproduct is a tracial von Neumann algebra and is equipped with a canonical trace $\tau((x_n)_{\omega})= \lim_{n\to \omega}\tau_n(x_n)$.

For a sequence of integers $d_{n}$, 
set
\[(\cM,\tau_{\omega})=\prod_{n\to\omega}(M_{d_{n}}(\bC),\tr).\]
We let
\[(L^{\infty}(\cL,u_{\omega}),u_{\omega})=\prod_{n\to\omega}(\ell^{\infty}(d_{n}),u_{d_{n}}),\]
and
\[\cS_{\omega}=\prod_{n\to\omega}(\Sym(d_{n}),d_{\Hamm}).\]
View $\ell^{\infty}(d_{n})\subseteq M_{d_{n}}(\bC)$ by identifying each function with the diagonal matrix whose entries  are 
$(f(1),\cdots,f(d_{n}))$,
and identify each permutation with its corresponding permutation matrix. In this way we can identify $L^{\infty}(\cL,u_{\omega})$ as a subalgebra of $\cM$ and $\cS_{\omega}$ as a subgroup of the unitary group of 
$\cM$. 
\begin{defn}

Let $(X,\nu,\cR)$ be a discrete, probability measure preserving equivalence relation. We say that $\cR$ is \emph{sofic} if there is a free ultrafilter $\omega$, a sequence of positive integers $(d_{n})_{n}$ and maps
\[\rho\colon L^{\infty}(X,\nu)\to L^{\infty}(\cL,u_{\omega}):=\prod_{n\to\omega}(\ell^{\infty}(d_{n}),u_{d_{n}}),\,\,\, \sigma\colon [\cR]\to \cS_{\omega}:=\prod_{n\to\omega}(\Sym(d_{n}),d_{\Hamm})\]
so that:
\begin{itemize}
    \item $\rho$ is a normal $*$-homomorphism,
    \item $\sigma$ is a homomorphism,
    \item $\tau_{\omega}(\rho(f)\sigma(\gamma))=\int_{\{x\in X:\gamma(x)=x\}}f\,d\nu$ for every $f\in L^{\infty}(X,\nu),\gamma\in [\cR]$,
    \item $\sigma(\gamma)\rho(f)\sigma(\gamma)^{-1}=\rho(f\circ \gamma^{-1})$ for every $f\in L^{\infty}(X,\nu),\gamma\in [\cR]$.
\end{itemize}

\end{defn}

Since $[\cR],L^{\infty}(X,\nu)$ are uncountable, this definition can be a bit unwieldy. We give a few equivalent definitions of soficity below for the readers convenience. Essentially all of this is either folklore or due to Elek-Lippner (see the proof \cite[Theorem 2]{ElekLip}) or P\u{a}unescu (\cite{PaunSofic}), we do not claim originality for these results.

\begin{prop} \label{prop: equivalent definition of a sofic rep}
Let $\cR$ be a discrete, probability measure-preserving relation over a standard probability space $(X,\nu)$. View $L^{\infty}(X,\nu)$ and $[[\cR]]$ as subsets of $L(\cR)$. 
\begin{enumerate}[(i)]
\item Let $\rho\colon L^{\infty}(X,\nu)\to L^{\infty}(\cL,u_{\omega})$, $\sigma\colon [\cR]\to \cS_{\omega}$ be a sofic approximation.  Then there is a trace-preserving $*$-homomorphism $\pi\colon L(\cR)\to \prod_{k\to\omega}(M_{k}(\bC),\tr)$ so that $\pi|_{L^{\infty}(X)}=\rho$, $\pi|_{[\cR]}=\sigma$. \label{item: the better defn}
\item Conversely, suppose that $\pi\colon L(\cR)\to \prod_{k\to\omega}(M_{k}(\bC),\tr)$ is a trace-preserving $*$-homomorphism with $\pi(L^{\infty}(X))\subseteq L^{\infty}(\cL,u_{\omega})$ and $\pi([\cR])\subseteq \cS_{\omega}$. Then the pair $(\rho,\sigma)$ given by $\rho=\pi|_{L^{\infty}(X,\nu)}$, $\sigma=\pi|_{[\cR]}$ is a sofic approximation. \label{item: better defn converse}
\item Let $D\subseteq L^{\infty}(X,\nu)$ be a subset which is closed under products,
and $G\subseteq [\cR]$ a countable subgroup. Suppose that $D$ is $G$-invariant, that $\Span(D)$ is weak$^{*}$-dense in $L^{\infty}(X)$ and that $Gx=[x]_{\cR}$ for almost every $x\in X$. Suppose that $\rho_{0}\colon D\to L^{\infty}(\cL,u_{\omega})$ and $\sigma_{0}\colon G\to \cS_{\omega}$ are such that:
\begin{itemize}
    \item $\rho_{0}(f_{1}f_{2})=\rho_{0}(f_{1})\rho_{0}(f_{2})$ for every $f_{1},f_{2}\in D$
    \item $\sigma_{0}(\gamma)\rho_{0}(f)\sigma_{0}(\gamma)^{-1}=\rho_{0}(f\circ \gamma^{-1})$ for every $f\in D,\gamma\in G$.
    \item $\tau_{\omega}(\rho_{0}(f)\sigma_{0}(\gamma))=\int_{\{x\in X:\gamma(x)=x\}}f\,d\nu$ for all $f\in D,\gamma\in G$.
\end{itemize}
Then there is a unique sofic approximation $(\rho,\sigma)$ of $\cR$ so that $\rho|_{D}=\rho_{0},$ $\sigma|_{G}=\sigma_{0}$.
\label{item: countable reduction sofic rep}
\item Let $\cA$ be an algebra of measurable sets in $X$, and let $G$ be a countable subgroup of $[\cR]$ with $Gx=[x]_{\cR}$ for almost every $x\in X$. Assume that $\cA$ is $G$-invariant and that the complete sigma-algebra generated by $\cA$ is the algebra of all $\nu$-measurable sets. Suppose that $\rho_{0}\colon \cA\to \Proj(L^{\infty}(\cL,u_{\omega}))$ and $\sigma_{0}\colon G\to \cS_{\omega}$ satisfies:
\begin{itemize}
    \item $\rho_{0}(B_{1}\cap B_{2})=\rho_{0}(B_{1})\rho_{0}(B_{2})$ and $u_{\omega}(\rho_{0}(B_{1}))=\nu(B_{1})$ for every $B_{1},B_{2}\in \cA$,
    \item $\sigma_{0}$ is a homomorphism and $\tau(\rho_{0}(B)\sigma_{0}(\gamma))=\nu(\{x\in B:\gamma(x)= x\})$ for every $\gamma\in G$, $B\in \cA$,
    \item $\sigma_{0}(\gamma)\rho_{0}(E)\sigma_{0}(\gamma)^{-1}=\rho_{0}(\gamma(B))$ for every $\gamma\in G,B\in\cA$.
\end{itemize}
Then there is a unique sofic approximation $\rho\colon L^{\infty}(X,\nu)\to L^{\infty}(\cL,u_{\omega})$, $\sigma\colon [\cR]\to S_{\omega}$ so that
$\rho|_{\cA}=\rho_{0}$, $\sigma|_{G}=\sigma_{0}$.
\label{item: sometimes people don't like addition}
\end{enumerate}

\end{prop}

\begin{proof}
Throughout, set
$(\cM,\tau_{\omega})=\prod_{k\to\omega}(M_{k}(\bC),\tr).$

(\ref{item: the better defn}): Let
\[M_{0}=\left\{\sum_{\gamma\in [\cR]}a(\gamma)\gamma:a\colon[\cR]\to L^{\infty}(X,\nu) \textnormal{ is finitely supported}\right\},\]
then $M_{0}$ is weak operator topology dense in $L(\cR)$ by definition. Moreover, for all $f_{1},\cdots,f_{n}\in L^{\infty}(X,\nu)$,$\gamma_{1},\cdots,\gamma_{n}\in [\cR]$ we have by the axioms of a sofic approximation:
\begin{align*}
  \left\|\sum_{j=1}^{n}\rho(f_{j})\sigma(\gamma_{j})\right\|_{2}^{2}=\sum_{i,j}\tau_{\omega}(\sigma(\gamma_{j})^{-1}\rho(\overline{f_{j}}f_{k})\sigma(\gamma_{k}))&=\sum_{i,j}\tau_{\omega}(\rho((\overline{f_{j}}f_{k})\circ \gamma_{j})\sigma(\gamma_{j}^{-1}\gamma_{k})) \\
  &=\sum_{i,j}\int_{\{x:\gamma_{j}^{-1}(\gamma_{k}(x))=x\}}(\overline{f_{j}}f_{k})\circ \gamma_{j}\,d\nu\\
&=\sum_{i,j}\tau([(\overline{f_{j}}f_{k})\circ \gamma_{j}\gamma_{j}^{-1}]\lambda(\gamma_{k}))\\
&=\sum_{i,j}\tau(\lambda(\gamma_{j})^{-1}\overline{f_{j}}f_{k}\gamma\lambda(\gamma_{k}))\\
&=\left\|\sum_{j=1}^{n}f_{j}\lambda(\gamma_{j})\right\|_{2}^{2}.
\end{align*}
Since $\|\cdot\|_{2}$ is a norm, the above calculation implies that $\sum_{j}\rho(f_{j})\sigma(\gamma_{j})=0$ if and only if $\sum_{j}f_{j}\lambda(\gamma_{j)}=0$. This implies that the map
$\pi_{0}\colon M_{0}\to \prod_{k\to\omega}(M_{k}(\bC),\tr)$
given by
\[\pi_{0}\left(\sum_{\gamma}f(\gamma)\lambda(\gamma)\right)=\sum_{\gamma}\rho(f(\gamma))\sigma(\gamma),\]
is a  well-defined linear map. It is direct to check from the definition of a sofic approximation that it is a trace-preserving $*$-homomorphism.
 Lemma \ref{lem: folklore ish} implies that $\pi_{0}$ has a unique extension to trace-preserving normal $*$-homomorphism $\pi$.

(\ref{item: better defn converse}): This is an exercise in understanding the definitions.

(\ref{item: countable reduction sofic rep}):
Let $A=\Span(D)$ so that $A$ is a weak operator topology dense $*$-subalgebra of  $L^{\infty}(X,\nu)$. 
Let
\[M_{0}=\left\{\sum_{g\in G}a(g)\lambda(g):a\in c_{c}(G,A)\right\},\]
then $M_{0}$ is $*$-subalgebra of $L(\cR)$.
As in (\ref{item: the better defn}) we 
 know there is a unique function
$\pi_{0}\colon M_{0}\to\cM$
satisfying
\[\pi_{0}\left(\sum_{g\in G}a(g)\lambda(g)\right)=\sum_{g\in G}\rho_{0}(a(g))\sigma_{0}(g) \textnormal{ for all $a\in c_{c}(G,A)$.}\]
Our hypothesis imply that $\pi_{0}$ is a trace-preserving $*$-homomorphism. Thus by Lemma \ref{lem: folklore ish} and items (\ref{item: the better defn}),(\ref{item: better defn converse}) it suffices to show  that
$M_{0}$ is weak$^{*}$-dense in $L(\cR)$. Let $M$ be the weak operator topology closure of $M_{0}$. Then
 \[M\supseteq \overline{A}^{weak^{*}}=L^{\infty}(X,\nu).\]
Since $Gx=[x]_{\cR}$ for almost every $x\in X$,  given $\gamma\in [[\cR]]$ we may find (not necessarily unique) disjoint sets $(B_{g})_{g\in G}$ so that $B_{g}\subseteq\{x\in \dom(\gamma):\gamma(x)=gx\}$ and with
\[\nu\left(\dom(\gamma)\Delta \bigsqcup_{g\in G}B_{g}\right)=0.\]
We leave it as an exercise to check that
\[\lambda(\gamma)=\sum_{g\in G}1_{B_{g^{-1}}}\lambda(g)\]
with the sum converging in the strong operator topology. Since we have already shown that $L^{\infty}(X,\nu)\subseteq M$, it follows that $[[\cR]]\subseteq M$.
Hence
\[L(\cR)=W^{*}([[\cR]])\subseteq M.\]

(\ref{item: sometimes people don't like addition}):
Let $D=\{1_{E}:E\in \cA\}$, then $\Span(D)$ is weak$^{*}$-dense in $L^{\infty}(X)$ since $\cA$ is generating. Now apply (\ref{item: countable reduction sofic rep}).

\end{proof}

Because Proposition \ref{prop: equivalent definition of a sofic rep} gives several equivalent definitions of soficity, we will often use the term \emph{sofic approximation} to any one of the kinds of maps in each item of this proposition. For example, a map $\pi\colon L(\cR)\to \prod_{k\to\omega}(M_{k}(\bC),\tau)$ satisfying the hypotheses of (\ref{item: the better defn}) will be called a sofic approximation.

Additionally, item (\ref{item: sometimes people don't like addition}) suggest a sequential version of a sofic approximation of an equivalence relation. Namely, we can consider sequences $\rho_{n}\colon \cA\to \mathcal{P}(\{1,\cdots,d_{n}\})$ and $\sigma_{n}\colon G\to \Sym(d_{n})$ so that:
\begin{itemize}
    \item $\sigma_{n}$ is an asymptotic homomorphism,
    \item $u_{d_{n}}(\rho_{n}(B_{1}\cap B_{2})\Delta (\rho_{n}(B_{1})\cap \rho_{n}(B_{2})))\to_{n\to\infty}0,$ for all $B_{1},B_{2}\in \cA$,
    \item $u_{d_{n}}(\{j\in \rho_{n}(B):\sigma_{n}(g)(j)=j\})\to_{n\to\infty} \nu(\{x\in E:gx=x\})$ for all $B\in \cA$, $g\in G$,
    \item $u_{d_{n}}(\rho_{n}(gB)\Delta(\sigma_{n}(g)\rho_{n}(B)\sigma_{n}(g)^{-1}))\to_{n\to\infty}0$ for all $g\in G$, $B\in \cA$.
\end{itemize}
This is sometimes taken as the definition of soficity. Our definition has the advantage of being canonical and not requiring a choice of $G,\cA$. However, this alternate definition is typically how one would check soficity of relations in specific examples, whereas ours is more abstract. Similarly, item (\ref{item: countable reduction sofic rep}), suggests a different definition of soficity. One could require a sequence of maps $\rho_{n}\colon D\to \ell^{\infty}(d_{n})$, $\sigma_{n}\colon G\to \Sym(d_{n})$ so that
\begin{itemize}
    \item $\|\rho_{n}(f_{1}f_{2})-\rho_{n}(f_{1})\rho_{n}(f_{2})\|_{2}\to_{n\to\infty}0$ for all $f_{1},f_{2}\in D$,
    \item $\sigma_{n}$ is an asymptotic homomorphism,
    \item $\frac{1}{d_{n}}\sum_{j:\sigma_{n}(g)(j)=j}\rho_{n}(f)(j)\to_{n\to\infty}\int_{\{x\in X:gx=x\}}f\,d\nu$ for all $f\in D,g\in G$.
    \item $\|\rho_{n}(f)\circ \sigma_{n}(g)^{-1}-\rho_{n}(\alpha_{g}(f))\|_{2}\to_{n\to\infty} 0$ for every $f\in D$,$g\in G$.
\end{itemize}
Here the $\|\cdot\|_{2}$-norms are with respect to the uniform measure and $\alpha_{g}(f)(x)=f(g^{-1}x)$.
We will refer to either of these as a sofic approximation sequence.

We now rephrase Theorem \ref{thm: main theorem intro}.
We use $m$ for the Lebesgue measure on $[0,1]$.

\begin{thm}[(Theorem \ref{thm: main theorem intro})]\label{thm: I said the real Bernoulli shift}
 Let $G$ be countable, discrete group and let $\sigma_{n}\colon G\to \Sym(d_{n})$ be a sequence of asymptotic homomorphisms. Fix $\omega\in\beta\bN\setminus \bN$, and set $\Theta=\IRS(\sigma_{\omega})$. Let $\cR$ be the orbit equivalence relation of the $\Theta$-Bernoulli action with base $([0,1],m)$. Let $\Xi\colon G\to [\cR]$ be given by $\Xi(g)(x)=gx$. Then there is a sofic approximation $(\rho,\widehat{\sigma})$ of $\cR$  so that $\widehat{\sigma}\circ \Xi=\sigma_{\omega}$.
\end{thm}





\section{The proof of Theorem \ref{thm: I said the real Bernoulli shift}}\label{sec: real Bern proof}

We now proceed to prove Theorem \ref{thm: I said the real Bernoulli shift}.
We remark that one can use \cite[Theorem 1.5]{RobinTDWeakEquiv} to prove Theorem \ref{thm: I said the real Bernoulli shift}. The essential idea behind such a proof is that given a sequence of approximate homomorphisms with IRS $\Theta$, the induced action on the Loeb measure space $G\actson (\cL,u_{\omega})$ also has IRS $\Theta$. One can use the L\"{o}wenheim–Skolem theorem in continuous model theory to build an action on a standard probability space which is a factor of this action and which still has IRS $\Theta$. Being an action on a standard probability space, such a factor is weakly contained in a $\Theta$-Bernoulli shift  by \cite[Theorem 1.5]{RobinTDWeakEquiv}. This implies that the $\Theta$-Bernoulli shift action $G\actson (X_{\Theta},\mu_{\Theta})$ with base $[0,1]$ (equipped with Lebesgue measure)  is weakly contained in $G\actson (\cL,u_{\omega})$.  Since $G\actson (\cL,u_{\omega})$ is an ultraproduct action, if $G\actson (X_{\Theta},\mu_{\Theta})$ is weakly contained in $G\actson (\cL,u_{\omega})$ it must actually be a factor of this action. Putting this altogether shows that any sequences of approximate homomorphisms extends to a sofic approximation of the $\Theta$-Bernoulli shift.

For the sake of concreteness, we have instead elected to give a direct probabilistic argument for the proof of Theorem \ref{thm: I said the real Bernoulli shift}.
The following is the main technical probabilistic lemma we need and is inspired by the proof of \cite[Theorem 8.1]{Bow}. Bowen's purpose in \cite{Bow} for such a result was to prove that the sofic entropy of Bernoulli shifts is the Shannon entropy of the base space. In this regard our proof should be compared with \cite[Theorem 9.1]{SewardSinai} which also says, in some sense, that the $\Theta$-Bernoulli shift with base $(X,\nu)$ is the ``largest entropy" action whose IRS is $\Theta$ and which is generated by the translates of an $X$-valued random variable with distribution $\nu$.

For ease of notation, if $X$ is a compact, metrizable space, $\nu\in \Prob(X)$, and $f\colon X\to \bC$ is bounded and Borel we often use $\nu(f)$ for $\int f\,d\nu$. 
Suppose $I$ is a set,  and $(X_{i})_{i\in I}$ are compact Hausdorff spaces. If $E\subseteq I$ is finite, and $(f_{i})_{i\in E}\in \prod_{i\in E}C(X_{i}),$ we define $\bigotimes_{i\in E}f_{i}\in C\left(\prod_{i\in I}X_{i}\right)$ by
\[\left(\bigotimes_{i\in E}f_{i}\right)(x)=\prod_{i\in E}f_{i}(x_{i}), \textnormal{ if $x=(x_{i})_{i\in I}\in \prod_{i\in I}X_{i}.$}\]

For a finite $F\subseteq G$ we define $\cE_{F}\in C(\{0,1\}^{G}$) by
\[\cE_{F}(a)=\prod_{g\in F}a(g).\]
Viewing $\Sub(G)\subseteq\{0,1\}^{G}$ we have $\cE_{F}\big|_{\Sub(G)}$ is the indicator function of $\{H:F\subseteq H\}$. If $Y$ is a compact, metrizable space we often slightly abuse notation and regard $\cE_{F}\in C(\{0,1\}^{G}\times Y)$ via the embedding $C(\{0,1\}^{G})\to C(\{0,1\}^{G}\times Y)$ given by
\[f\mapsto ((a,y)\mapsto f(a)).\]

\begin{lemm}\label{lem:building the embedding via prob}
Let $\sigma_{n}\colon G\to \Sym(d_{n})$ be a sequence of asymptotic homomorphisms and $\omega\in \beta\bN\setminus\bN$. Set $\Theta=\IRS(\sigma_{\omega})$. Let $\mu_{\Theta}$ be the $\Theta$-Bernoulli measure on $\Sub(G)\times [0,1]^{G}$ as in Definition \ref{defn: real Bernoulli shift}. View $\Sub(G)\times [0,1]^{G}\subseteq (\{0,1\}\times [0,1])^{G}$ by identifying each subgroup with its indicator function.
For $x\in [0,1]^{d_{n}}$, define $\phi_{x}\colon [d_{n}]\to (\{0,1\}\times [0,1])^{G}$ by
\[\phi_{x}(j)(g)=(1_{\{j\}}(\sigma_{n}(g)(j)),x(\sigma_{n}(g)^{-1}(j))).\]
Then:
\begin{enumerate}[(i)]
\item \label{item: on average wk* convergence}
in the weak$^{*}$-topology
\[\lim_{n\to\omega}\int_{[0,1]^{d_{n}}}(\phi_{x})_{*}(u_{d_{n}})dx=\mu_{\Theta}.\]
\item \label{item:variance argument} Given $f\in C(\{0,1\}^{G}\times [0,1]^{G})$, set $J_{n,f}=\int_{[0,1]^{d_{n}}}(\phi_{x})_{*}(u_{d_{n}})(f)\,dx$. Then:
\[\lim_{n\to\infty}\int_{[0,1]^{d_{n}}}|(\phi_{x})_{*}(u_{d_{n}})(f)-J_{n,f}|^{2}\,dx=0.\]
\item \label{item:asymptotic equivariance}
For $g\in G$, define $\alpha_{g}\in B(C((\{0,1\}\times [0,1])^{G}))$ by $(\alpha_{g}(f))(H,x)=f(g^{-1}Hg,g^{-1}x)$. Then for any $f\in C(\{0,1\}\times [0,1])^{G})$, $g\in G:$
\[\lim_{n\to\infty}\sup_{x\in [0,1]^{d_{n}}}\|f\circ \phi_{x}\circ \sigma_{n}(g)^{-1}-\alpha_{g}(f)\circ \phi_{x}\|_{\ell^{2}(u_{d_{n}})}=0.\]
\end{enumerate}

\end{lemm}

\begin{proof}
Throughout, we use $G/H$ for the space of right cosets of $H$ in $G$. For $F\subseteq G$ finite, we set
\[\varsigma_{F,n}=\bigcap_{g\in F}\{j:\sigma_{n}(g)(j)=j\}\]
(\ref{item: on average wk* convergence}):
By the Riesz representation theorem, we can define $\eta\in \Prob((\{0,1\}\times [0,1])^{G})$ by
\[\eta(f)=\lim_{n\to\omega}\int_{[0,1]^{d_{n}}}(\phi_{x})_{*}(u_{d_{n}})(f)dx \textnormal{ for all $f\in C((\{0,1\}\times [0,1])^{G})$}.\]
 Let
\[\digamma=\left\{\left(\bigotimes_{g\in E}f_{g}\right)\cE_{F}:E,F\subseteq G
 \textnormal{ are finite}, (f_{g})_{g\in E}\in C([0,1])^{E},f_{g}\geq 0 \textnormal{ for all $g\in E$ }\right\}.\]
 By Stone-Weierstrass, $\Span(\digamma)$ is norm dense in $C((\{0,1\}\times [0,1])^{G})$, and by the Riesz representation theorem to show that $\eta=\mu_{\Theta}$ it suffices to show that they have the same integral against any element of $\digamma$. Fix $f=\left(\bigotimes_{g\in E}f_{g}\right)\cE_{F}\in \digamma$.
Let $\Omega_{n}$ be the set of $j\in [d_{n}]$ so that for all $1\leq k\leq (2023)!$ and for all $g_{1},\cdots,g_{k}\in E$ and all $s_{1},\cdots,s_{k}\in \{-1,1\}$ we have that
\[\sigma_{n}(g_{1}^{s_{1}}\cdots g_{k}^{s_{k}})(j)=\sigma_{n}(g_{1})^{s_{1}}\cdots \sigma_{n}(g_{k})^{s_{k}}(j).\]
Since $\sigma_{n}$ is an approximate homomorphism $u_{d_{n}}(\Omega_{n})\to 1$. 
Hence:
\begin{align*}
 \int_{[0,1]^{d_{n}}}(\phi_{x})_{*}(u_{d_{n}})(f)\,dx&=\frac{1}{d_{n}}\sum_{j\in \varsigma_{F,n}}\int_{[0,1]^{d_{n}}}\prod_{g\in E}f_{g}(x(\sigma_{n}(g)^{-1}(j)))\,dx\\
 &=u_{d_{n}}(\Omega_{n}^{c})\frac{1}{|\Omega_{n}|^{c}}\sum_{j\in \varsigma_{F,n}\setminus \Omega_{n}}\int_{[0,1]^{d_{n}}}\prod_{g\in E}f_{g}(x(\sigma_{n}(g)^{-1}(j)))\,dx\\
 &+u_{d_{n}}(\Omega_{n})\frac{1}{|\Omega_{n}|}\sum_{j\in \varsigma_{F,n}\cap \Omega_{n}}\int_{[0,1]^{d_{n}}}\prod_{g\in E}f_{g}(x(\sigma_{n}(g)^{-1}(j)))\,dx.
\end{align*}
Since $u_{d_{n}}(\Omega_{n})\to 1$, and the $f_{g}$ are bounded, the first term tends to zero. It thus remains to analyze the second term. For $j\in\Omega_{n}$, 
\begin{align*}
   &\int_{[0,1]^{d_{n}}}\prod_{g\in E}f_{g}(x(\sigma_{n}(g)^{-1}(j)))\,dx=\prod_{s\in \sigma_{n}(E^{-1})(j)}\int_{0}^{1}\prod_{g\in E:\sigma_{n}(g^{-1})(j)=s}f_{g}(x)\,dx \\
   &=\prod_{g\in E}\left(\int_{0}^{1}\prod_{k\in E:\sigma_{n}(k^{-1})(j)=\sigma_{n}(g^{-1})(j)}f_{k}(x)\,dx\right)^{\frac{1}{|\{a\in E:\sigma_{n}(a^{-1})(j)=\sigma_{n}(g^{-1})(j)\}|}},
\end{align*}
where in the last step we use that $f_{g}\geq 0$ to make sense of the fractional power (which is there to account for over-counting). Because $j\in \Omega_{n}$:
\begin{align*}
    &\int_{[0,1]^{d_{n}}}\prod_{g\in E}f_{g}(x(\sigma_{n}(g)^{-1}(j)))\,dx=\prod_{g\in E}\left(\int_{0}^{1}\prod_{k\in E}1_{\{j\}}(\sigma_{n}(gk^{-1})(j))f_{k}(x)\,dx\right)^{\frac{1}{|\{a\in E:\sigma_{n}(ga^{-1})(j)=j\}|}}\\
\end{align*}
Since $u_{d_{n}}(\Omega_{n})\to 1$, we have altogether shown that
\begin{align*}
  &\lim_{n\to\omega}\int_{[0,1]^{d_{n}}}(\phi_{x})_{*}(u_{d_{n}})(f)\,dx=\\
  &\lim_{n\to\omega}\frac{1}{d_{n}}\sum_{j\in [d_{n}]}\prod_{b\in F}1_{\{j\}}(\sigma_{n}(b)(j))\prod_{g\in E}\left(\int_{0}^{1}\prod_{k\in E}1_{\{j\}}(\sigma_{n}(gk^{-1})(j))f_{k}(x)\,dx\right)^{\frac{1}{|\{a\in E:\sigma_{n}(ga^{-1})(j)=j\}|}}.
\end{align*}
It is direct to verify that the function $\Phi_{f}\colon \{0,1\}^{G}\to [0,+\infty)$ given by
\[\Phi_{f}(y)=\left(\prod_{b\in F}y(b)\right)\prod_{g\in E}\left(\int_{0}^{1}\prod_{k\in E}y(gk^{-1})f_{k}(x)\,dx\right)^{\frac{1}{|\{a\in E:y(ga^{-1})=1\}|}}\]
is continuous (in fact, it is locally constant).
The definition of $\IRS(\sigma_{\omega})$ thus proves that:
\[\lim_{n\to\omega}\int_{[0,1]^{d_{n}}}(\phi_{x})_{*}(u_{d_{n}})(f)\,dx=
\int \Phi_{f}(H)\,d\Theta(H).\]
By definition,
\[\left(y\mapsto \prod_{b\in F}y(b)\right)=\cE_{F}\]
moreover, if $H\in \Sub(G)$, then $1_{H}(ga^{-1})=1$ if and only if $Ha=Hg$. Thus, viewing $\Sub(G)\subseteq\{0,1\}^{G}$, we have 
\[\Phi_{f}(H)=\cE_{F}(H)\prod_{g\in E}\left(\int_{0}^{1}\prod_{k\in E}1_{H}(gk^{-1})f_{k}(x)\,dx\right)^{\frac{1}{|\{a\in E:Ha=Hg\}|}}, \textnormal{ for all $H\in \Sub(G)$}\]
So
\[\lim_{n\to\omega}\int_{[0,1]^{d_{n}}}(\phi_{x})_{*}(u_{d_{n}})(f)\,dx=\int_{\Sub(G)}\cE_{F}(H)\prod_{g\in E}\left(\int_{0}^{1}\prod_{k\in E}1_{H}(gk^{-1})f_{k}(x)\,dx\right)^{\frac{1}{|\{a\in E:Ha=Hg\}|}}\,d\Theta(H).\]
Note that 
\[\prod_{c\in HE/H}\left(\int_{0}^{1}\prod_{g\in c\cap E}f_{g}(x)\,dx\right)=\prod_{g\in E}\left(\int_{0}^{1}\prod_{k\in E}1_{H}(gk^{-1})f_{k}(x)(x)\,dx\right)^{\frac{1}{|\{a\in E:Ha=Hg\}|}},\]
with the last step following because the inverse image of $g$ under the map $E\mapsto G/H,g\mapsto Hg$ has cardinality $|\{a\in E:Ha=Hg\}|.$
Thus
\[\lim_{n\to\omega}\int_{[0,1]^{d_{n}}}(\phi_{x})_{*}(u_{d_{n}})(f)\,dx=\int_{\Sub(G)}\cE_{F}(H)\prod_{c\in HE/H}\left(\int_{0}^{1}\prod_{g\in c\cap E}f_{g}(x)\,dx\right)\,d\Theta(H)\]
Note that for $H\in \Sub(G)$, we have that 
\[\int_{[0,1]^{G}}\prod_{g\in E}f_{g}(x(g))\,dm^{\otimes G/H}(x)=\prod_{c\in HE/H}\left(\int_{0}^{1}\prod_{g\in c\cap E}f_{g}(x)\,dx\right),\]
where we view $m^{\otimes G/H}$  as a probability measure on $[0,1]^{G}$ which is supported on the $x\in [0,1]^{G}$ which are constant on right $H$-cosets. So
\[\int \cE_{F}(H)\prod_{g\in E}f_{g}(x(g))\,d\mu_{\Theta}(H,x)=\int_{\Sub(G)}\cE_{F}(H)\prod_{c\in HE/H}\left(\int_{0}^{1}\prod_{g\in c\cap E}f_{g}(x)\,dx\right)\,d\Theta(H)\]
Altogether, this shows that 
\begin{align*}
    \lim_{n\to\omega}\int_{[0,1]^{d_{n}}}(\phi_{x})_{*}(u_{d_{n}})(f)\,dx&=\int \cE_{F}(H)\prod_{g\in E}f_{g}(x(g))\,d\mu_{\Theta}(H,x)\\
    &=\int f\,d\mu_{\Theta},
\end{align*}
as  desired.

(\ref{item:variance argument}):
Define
$T_{n}\colon C((\{0,1\}\times [0,1])^{G})\to L^{2}([0,1]^{d_{n}})$
by
\[(T_{n}\zeta)(x)=(\phi_{x})_{*}(u_{d_{n}})(\zeta)-J_{n,\zeta}.\]
Giving $C((\{0,1\}\times [0,1])^{G})$ the supremum norm we have the operator norm of $T_{n}$ satisfies $\|T_{n}\|\leq 2$. This uniform estimate implies that $\{\zeta\in C((\{0,1\}\times [0,1])^{G}):\|T_{n}\zeta\|_{2}\to_{n\to\omega}0\}$ is a closed, linear subspace. So as in (\ref{item: on average wk* convergence}), it suffices to verify the desired statement for $f=\left(\bigotimes_{g\in E}f_{g}\right)\cE_{F}\in B$, where $(f_{g})_{g\in E}\in C([0,1])^{G}$ and $E,F\subseteq G$ are finite.
By direct computation,
\begin{equation}\label{eqn: its a variance}
\int_{[0,1]^{d_{n}}}|(\phi_{x})_{*}(u_{d_{n}})(f)-J_{n,f}|^{2}\,dx=\int_{[0,1]^{d_{n}}}|(\phi_{x})_{*}(u_{d_{n}})|^{2}\,dx-|J_{n,f}|^{2}.
\end{equation}
Moreover,
\[\int_{[0,1]^{d_{n}}}|(\phi_{x})_{*}(u_{d_{n}})|^{2}\,dx=\frac{1}{d_{n}^{2}}\sum_{j,k\in \varsigma_{F,n}}\int_{[0,1]^{d_{n}}}\prod_{g,h\in E}f_{g}(x(\sigma_{n}(g)^{-1}(j)))\overline{f_{h}(x(\sigma_{n}(h)^{-1}(j)))}\,dx.\]
We can rewrite this as
\begin{align} \label{eqn: you know I had to do it}
  &\frac{1}{d_{n}^{2}}\sum_{\substack{j,k\in \varsigma_{F,n},\\ \sigma_{n}(E)^{-1}(j)\cap \sigma_{n}(E)^{-1}(k)\ne \varnothing}}\int_{[0,1]^{d_{n}}}\prod_{g,h\in E}f_{g}(x(\sigma_{n}(g)^{-1}(j)))\overline{f_{h}(x(\sigma_{n}(h)^{-1}(k)))}\,dx\\
  &+\frac{1}{d_{n}^{2}}\sum_{\substack{j,k\in \varsigma_{F,n},\\ \sigma_{n}(E)^{-1}(j)\cap \sigma_{n}(E)^{-1}(k)= \varnothing}}\int_{[0,1]^{d_{n}}}\prod_{g,h\in E}f_{g}(x(\sigma_{n}(g)^{-1}(j)))\overline{f_{h}(x(\sigma_{n}(h)^{-1}(k)))}\,dx  \nonumber
\end{align}
For a fixed $j$, the set of $k$ for which $\sigma_{n}(E)^{-1}(j)\cap \sigma_{n}(E)^{-1}(k)\ne \varnothing$ has cardinality at most $|E|^{2}$. So the first term is bounded by
\[\frac{1}{d_{n}}|E|^{2}\prod_{g\in E}\|f_{g}\|_{\infty}^{2}\to_{n\to\omega}0.\]
For $n\in \bN$, define $\Psi_{n,f}\colon \{1,\cdots,d_{n}\}\to \bR$ by
\[\Psi_{n,f}(j)=\int_{[0,1]^{d_{n}}}\prod_{g\in E}f_{g}(x(\sigma_{n}(g)^{-1}(j)))\,dx\]
The second term in (\ref{eqn: you know I had to do it}) is:
\[\frac{1}{d_{n}^{2}}\sum_{\substack{j,k\in \varsigma_{F,n},\\ \sigma_{n}(E)^{-1}(j)\cap \sigma_{n}(E)^{-1}(k)= \varnothing}}\Psi_{n,f}(j)\overline{\Psi_{n,f}(k)}\]
Thus
\begin{equation}\label{eqn:collecting variance terms I}
\lim_{n\to\omega}\int_{[0,1]^{d_{n}}}|(\phi_{x})_{*}(u_{d_{n}})|^{2}\,dx=\lim_{n\to\omega}\frac{1}{d_{n}^{2}}\sum_{\substack{j,k\in \varsigma_{F,n},\\ \sigma_{n}(E)^{-1}(j)\cap \sigma_{n}(E)^{-1}(k)=\varnothing}}\Psi_{n,f}(j)\overline{\Psi_{n,f}(k)}.
\end{equation}
By another direct computation,
\[|J_{n,f}|^{2}=\frac{1}{d_{n}^{2}}\sum_{j,k\in \varsigma_{F,n}}\Psi_{n,f}(j)\overline{\Psi_{n,f}(k)}.\]
By the same estimates as above, 
\[\lim_{n\to\omega}|J_{n,f}|^{2}=\lim_{n\to\omega}\frac{1}{d_{n}^{2}}\sum_{\substack{j,k\in \varsigma_{F,n},\\ \sigma_{n}(E)^{-1}(j)\cap \sigma_{n}(E)^{-1}(k)=\varnothing}}\Psi_{n,f}(j)\overline{\Psi_{n,f}(k)}\]
 Combining this with (\ref{eqn: its a variance}),(\ref{eqn:collecting variance terms I})  shows that
 \[\lim_{n\to\omega}\int_{[0,1]^{d_{n}}}|(\phi_{x})_{*}(u_{d_{n}})(f)-J_{n,f}|^{2}\,dx=0.\]
 Since this is true for every free ultrafilter, we have proven (\ref{item:variance argument}).

 (\ref{item:asymptotic equivariance}):
 Since $f$ is continuous, we can find an $C\geq 0$ with $|f|\leq C$. View $\{0,1\}\times [0,1]\subseteq \bR^{2}$ and give $\bR^{2}$ the norm $\|(t,s)\|=\sqrt{t^{2}+s^{2}}$.  Given $\varepsilon>0$, compactness of $(\{0,1\}\times [0,1])^{G}$ and continuity of $f$ imply that we may find a finite $E\subseteq G$ and a $\delta>0$ so that if $x,y\in (\{0,1\}\times  [0,1])^{G}$ and
 \[\|x(g)-y(g)\|<\delta \textnormal{ for all $g\in E$},\]
 then
 \[|f(x)-f(y)|<\varepsilon.\]
 Then, for any $x\in [0,1]^{d_{n}}$ we have
 \begin{align*}\|f\circ \phi_{x}\circ \sigma_{n}(g)^{-1}&-\alpha_{g}(f)\circ \phi_{x}\|_{\ell^{2}(u_{d_{n}})}^{2}\\
 &\leq\varepsilon^{2}+C^{2}u_{d_{n}}(\{j\in [d_{n}]:g^{-1}\phi_{x}(j)\big|_{E}\ne \phi_{x}(\sigma_{n}(g)^{-1}(j))\big|_{E}\})\\
 &\leq \varepsilon^{2}+C^{2}u_{d_{n}}\left(\bigcup_{h\in E}\{j:\sigma_{n}(gh)^{-1}(j)\ne \sigma_{n}(h)^{-1}\sigma_{n}(g)^{-1}(j)\}\right).
 \end{align*}
 Since $\sigma_{n}$ is an asymptotic homomorphism the second term tends to zero as $n\to\infty$.
\end{proof}

While technical, Lemma \ref{lem:building the embedding via prob} has all the tools to prove Theorem \ref{thm: I said the real Bernoulli shift}. Indeed, as we now show,  Lemma \ref{lem:building the embedding via prob} essentially says that a random choice of $x\in [0,1]^{d_{n}}$ will produce a sofic approximation of the appropriate equivalence relation.

\begin{proof}[Proof of Theorem \ref{thm: I said the real Bernoulli shift}]
For notation, define $S\colon [d_{n}]\to \{0,1\}^{G}$ by
$S(j)(g)=1_{\{j\}}(\sigma_{n}(g)(j)).$
Set
$\Theta_{n}=S_{*}(u_{d_{n}}).$
For $n\in \bN,f\in C((\{0,1\}\times [0,1])^{G})$, let $J_{n,f}=\int_{[0,1]^{d_{n}}} (\phi_{x})_{*}(u_{d_{n}})(f)\,dx$.
Fix a countable, dense set $D\subseteq C((\{0,1\}\times [0,1])^{G})$ which is closed under products and is $G$-invariant.  Write $D=\bigcup_{k=1}^{\infty}D_{k}$  and $G=\bigcup_{n=1}^{\infty}F_{k}$ where $D_{k},F_{k}$ are increasing sequences of finite sets. For $k\in\bN$, let $L_{k}$ be the set of $n\in \bN$ so that
\begin{itemize}
      \item $n\geq k$
    \item $\sum_{f\in D_{k},g\in F_{k}}|\mu_{\Theta}(f\cE_{g})-J_{n,f\cE_{g}}|<2^{-k}$
    \item $\sum_{f\in D_{k},g\in F_{k}}\left(\int_{[0,1]^{d_{n}}}\left|(\phi_{x})_{*}(u_{d_{n}})(f\cE_{g})-J_{n,f\cE_{g}}\right|^{2}\,dx\right)^{1/2}<2^{-k}$,
    \item $\sup_{x\in [0,1]^{d_{n}}}\sum_{g\in F_{k},f\in D_{k}}\|f\circ \phi_{x}\circ\sigma_{n}(g)^{-1}-\alpha_{g}(f)\circ \phi_{x}\|_{\ell^{2}(u_{d_{n}})}<2^{-k}.$
\end{itemize}
Then $L_{k}$ is a decreasing sequence of sets and Lemma \ref{lem:building the embedding via prob} implies that $L_{k}\in \omega$ for all $k\in \bN$, and $\bigcap_{k}L_{k}=\varnothing$ by the first bullet point. Set $L_{0}=\bN\setminus L_{1}$. For $n\in \bN$ let $k(n)$ be such that $n\in L_{k(n)}\setminus L_{k(n)+1}$.
For $n\in \bN$, let $\Omega_{n}$ be the set of $x\in [0,1]^{d_{n}}$ so that
\begin{equation}\label{eqn:last sofic approx}\sum_{f\in D_{k(n)},g\in F_{k(n)}}|(\phi_{x})_{*}(u_{d_{n}})(f\cE_{g})-J_{n,f\cE_{g}}|<2^{-k/2},
\end{equation}
Then for $n\in L_{k}$ we have $m^{\otimes d_{n}}(\Omega_{n}^{c})\leq 2^{-k/2}$. Hence if $k\geq (2023)!$ and $n\in L_{k}$, we may choose an $x_{n}\in \Omega_{n}$. Let $x_{n}$ be defined arbitrarily for $n\in \bN\setminus L_{(2023)!}$.
 Define 
\[\rho_{0}\colon D\to \prod_{k\to\omega}(\ell^{\infty}(d_{n}),u_{d_{n}}))\]
by
$\rho_{0}(f)=(f\circ \phi_{x_{k(n)}})_{n\to\omega}$. 
Then $\rho_{0}$  preserves products.
Recall that we view $\ell^{\infty}(d_{n})$ as a subalgebra of $M_{d_{n}}(\bC)$ by identifying each function with the corresponding diagonal matrix. Let $f\in D,g\in G$.
By the fact that $\sigma_{n}$ is an asymptotic homomorphism and the fourth bullet point above, we know that 
\[\sigma_{\omega}(g)\rho_{0}(f)\sigma_{\omega}(g)^{-1}=\rho_{0}(\alpha_{g}(f)).\]
 By our choice of $x_{n}$, 
\begin{align*}
\tau_{\omega}(\rho_{0}(f)\sigma_{\omega}(g))=\lim_{n\to\omega}\tr(f\circ \phi_{x_{k(n)}}\sigma_{n}(g))&=\lim_{n\to\omega}\frac{1}{d_{n}}\sum_{j}f(\phi_{x_{k(n)}}(j))1_{\{j\}}(\sigma_{n}(g)(j))\\
&=\lim_{n\to\omega}J_{n,f\cE_{g}},
\end{align*}
where in the last step we use (\ref{eqn:last sofic approx}). By Lemma \ref{lem:building the embedding via prob} this last limit is
\[\mu_{\Theta}(f\cE_{g})=\int_{\{(H,x):g\in H\}}f(H,x)\,d\mu_{\Theta}(H,x).\]
Since $([0,1],m)$ is atomless, and $G$ is countable
\[\mu_{\Theta}\left(\bigcup_{g_{1},g_{2}\in G}\{(H,x):Hg_{1}\ne Hg_{2}\textnormal{ and }x(g_{1})=x(g_{2})\}\right)=0.\]
Thus for every $g\in G$:
\[\mu_{\Theta}(\{(H,x):(gHg^{-1},gx)=(H,x)\}\Delta\{(H,x):g\in H\})=0.\]
So
\[\tau_{\omega}(\rho_{0}(f)\sigma_{\omega}(g))=\int_{\{(H,x):(gHg^{-1},gx)=(H,x)\}}f\,d\mu_{\Theta}.\]
Thus Proposition \ref{prop: equivalent definition of a sofic rep} (\ref{item: countable reduction sofic rep}) implies that implies there is a unique sofic approximation $\pi\colon L(\cR)\to \prod_{n\to\omega}(M_{d_{n}}(\bC),\tr)$ so that $\pi\circ \Xi=\sigma_{\omega}$ and $\pi|_{D}=\rho_{0}$. Set $\rho=\pi|_{L^{\infty}(X)}$ and $\widehat{\sigma}=\pi|_{[\cR]}$. By Proposition \ref{prop: equivalent definition of a sofic rep} (\ref{item: better defn converse}) we have that $(\rho,\widehat{\sigma})$ is a sofic approximation. The fact that $\widehat{\sigma}\circ \Xi=\sigma_{\omega}$ is true by construction.

\end{proof}

\section{Applications of Theorem \ref{thm: main theorem intro}} \label{sec: applications}

One of the applications of the Theorem we 
 wish to highlight is the following result which is a reformulation of the Newman-Sohler Theorem \cite{NSStable} (see \cite[Theorem 5]{ELEK20122593} for a statement of the Newman-Sohler theorem which is closer to our language).

\begin{thm}\label{thm:equivalence of IRS means conjugacy}
  Let $G$ be an amenable group, and $\sigma_{n},\psi_{n}\colon G\to \Sym(d_{n})$ two sequences of approximate homomorphisms. Let $\omega\in\beta\bN\setminus\bN$. Then $\sigma_{\omega}$ is conjugate to $\psi_{\omega}$ if and only if $\IRS(\sigma_{\omega})=\IRS(\psi_{\omega})$.
\end{thm}


\begin{proof}

Let $\Theta=\IRS(\sigma_{\omega})=\IRS(\psi_{\omega})$, and let $\cR$ be the orbit equivalence relation of the $\Theta$-Bernoulli shift with base $([0,1],m)$. Let $\Xi$ be given as in Theorem \ref{thm: I said the real Bernoulli shift}. By Theorem \ref{thm: I said the real Bernoulli shift}, we can find sofic approximations $(\rho_{j},\widehat{\sigma}_{j}),j=1,2$ so that $\widehat{\sigma}_{1}\circ \Xi=\sigma_{\omega}$, $\widehat{\sigma}_{2}\circ \Xi=\psi_{\omega}$. Since $G$ is amenable, we know by \cite{OWAnnounce}, \cite[II \S 3]{OrnWeiss} (see also \cite{CFW})  that $\cR$ is hyperfinite. Thus, by \cite[Proposition 1.20]{PaunSofic} there is a $\chi\in \cS_{\omega}$ with
\[\chi\widehat{\sigma}_{1}(\alpha)\chi^{-1}=\widehat{\sigma}_{2}(\alpha) \textnormal{ for all $\alpha\in [\cR]$.}\]
In particular,
\[\chi \sigma_{\omega}(g)\chi^{-1}=\chi\widehat{\sigma}_{1}(\Xi(g))\chi^{-1}=\widehat{\sigma}_{2}(\Xi(g))=\psi_{\omega}(g)\]
for all $g\in G$.

\end{proof}

We remark that Theorem \ref{thm:equivalence of IRS means conjugacy} recovers the result of Kerr-Li \cite[Lemma 4.5]{KLi2} and Elek-Szabo \cite{elekxzabo} on uniqueness of sofic approximations of amenable groups up to asymptotic conjugacy. Indeed, sofic approximations correspond to the cases $\IRS(\sigma_{\omega})=\IRS(\psi_{\omega})=\{1\}$. 
Another reformulation of this result is as follows.

\begin{cor}[Theorem 3.12 of \cite{APStable}] \label{cor:AP stability}
Let $G$ be an amenable group, and $\psi_{n},\sigma_{n}\colon G\to \Sym(d_{n})$ be approximate homomorphisms and fix a free ultrafilter $\omega\in\beta\bN\setminus\bN$. 
\begin{enumerate}
\item \label{item:AP ultrafilter}
Then  $(\psi_{\omega}),(\sigma_{\omega})$ are conjugate if and only if for all finite $F\subseteq G$
\[\lim_{n\to\omega}u_{d_{n}}\left(\bigcap_{g\in F}\{j:\psi_{n}(g)(j)=j\}\right)=\lim_{n\to\omega}u_{d_{n}}\left(\bigcap_{g\in F}\{j:\sigma_{n}(g)(j)=j\}\right)\]
\item  In particular, $(\sigma_{n})$,$(\psi_{n})_{n}$ are asymptotically conjugate as $n\to\infty$ if and only if for all finite $F\subseteq G$
 \[\lim_{n\to\infty}\left|\frac{1}{d_{n}}|\{j:\sigma_{n}(g)(j)=j \textnormal{ for all $g\in F$}\}|-\frac{1}{d_{n}}|\{j:\psi_{n}(g)(j)=j \textnormal{ for all $g\in F$}\}|\right|=0\]
 \label{item:AP asymptotic}
\end{enumerate}

\end{cor}

\begin{proof}

(\ref{item:AP ultrafilter}):
The forward implication is an exercise. For the reverse,
let $\Theta_{1}=\IRS(\sigma_{\omega}),\Theta_{2}=\IRS(\psi_{\omega})$. For every finite $F\subseteq G$ we have
\begin{align*}
   \Theta_{1}(\{H:F\subseteq H\})=\int \cE_{F}\,d\Theta_{1}&=\lim_{n\to\omega}\frac{1}{d_{n}}|\{j:\sigma_{n}(g)(j)=j\textnormal{ for all $g\in F$}\}|\\
&=\lim_{n\to\omega}\frac{1}{d_{n}}|\{j:\psi_{n}(g)(j)=j\textnormal{ for all $g\in F$}\}|\\
&=\int \cE_{F}\,d\Theta_{2}=\Theta_{2}(\{H:F\subseteq H\}).
\end{align*}
The $*$-subalgebra
$A_{0}=\Span\{\mathcal{E}_{F}:F\subseteq G\textnormal{ is finite}\}$
of $C(\Sub(G))$ contains $1$ and separates points, and is thus dense by Stone-Weierstrass.
It follows by the density and the Riesz representation theorem that 
 $\Theta_{1}=\Theta_{2}$. 
 The result now follows by Theorem \ref{thm:equivalence of IRS means conjugacy}.

 (\ref{item:AP asymptotic}):
 The forward implication is an exercise. For the reverse,
 suppose that  \[\lim_{n\to\infty}\left|\frac{1}{d_{n}}\left|\{j:\sigma_{n}(g)(j)=j\textnormal{ for all $g\in F$}\}\right|-\frac{1}{d_{n}}\left|\{j:\psi_{n}(g)(j)=j\textnormal{ for all $g\in F$}\}\right|\right|=0\]
for every finite $F\subseteq G$. To show that $\psi_{n},\sigma_{n}$ are conjugate it suffices by a diagonal argument to show that for every finite $E\subseteq G$
\[\lim_{n\to\infty}\inf_{\chi\in \Sym(d_{n})}\sum_{g\in E}d_{\Hamm}(\chi\sigma_{n}(g)\chi^{-1},\psi_{n}(g))=0.\]

If this does not hold, then there is an increasing sequence $n_{1}<n_{2}<\cdots$ and an $\varepsilon>0$ with
\[\inf_{\chi\in \Sym(d_{n})}\sum_{g\in E}d_{\Hamm}(\chi\sigma_{n_{k}}(g)\chi^{-1},\psi_{n_{k}}(g))\geq \varepsilon.\]
Let $\omega\in\beta\bN\setminus\bN$ with $\{n_{k}:k\in \bN\}\in \omega.$  By Corollary \ref{cor:AP stability} and our hypothesis, we see $(\sigma_{\omega}),(\psi_{\omega})$ are conjugate. Then there is a $\chi=(\chi_{n})_{n\to\omega}\in\prod_{n\to\omega}(\Sym(d_{n}),d_{\Hamm})$ so that $\chi\sigma_{\omega}(g)\chi^{-1}=\psi_{\omega}(g)$ for all $g\in G$.

In particular,
\[L=\left\{n:\sum_{g\in E}d_{\Hamm}(\chi_{n_{k}}\sigma_{n_{k}}(g)\chi_{n_{k}}^{-1},\psi_{n_{k}}(g))<\varepsilon\right\}\]
is in $\omega$ and by our choice of $n_{k},\omega$ it follows that $L\in \omega^{c}$. This contradicts $\omega$ being a filter.

\end{proof}

We remark that our proof of (\ref{item:AP ultrafilter}) goes through the fact that for a group $G$ any $\Theta\in \IRS(G)$ is uniquely determined by
\[(\Theta(\{H:F\subseteq H\}))_{F\subseteq G \textnormal{ finite}}\]
thus the data of the IRS and that of the action trace as given in \cite{APStable} are the same.

At this stage we need the following two propositions, for additional applications. Recall that if $\sigma_{n}\colon G\to\Sym(d_{n})$ is a sofic approximation and there is a $\Theta\in \Sub(G)$ with $\IRS(\sigma_{\omega})=\Theta$ for all $\omega\in \beta\bN\setminus\bN$, then we say that $(\sigma_{n})_{n}$ has \emph{stabilizer type $\Theta$.} Recall that we do \emph{not} define stabilizer type if $\IRS(\sigma_{\omega})$ is different for different choices of $\omega$.

\begin{prop}\label{prop: permanence}
Let $G$ be a sofic group and $\sigma_{n}\colon G\to \Sym(d_{n})$ approximate homomorphisms with stabilizer type $\Theta\in \IRS(G)$.
\begin{enumerate}[(i)]
\item \label{item: perturbation} If $\widetilde{\sigma}_{n}\colon G\to \Sym(d_{n})$ are maps so that
\[\lim_{n\to\infty}d_{\Hamm}(\sigma_{n}(g),\widetilde{\sigma}_{n}(g))=0,\textnormal{ for all $g\in G$},\]
then $\widetilde{\sigma}_{n}$ are approximate homomorphisms with stabilizer type $\Theta$.
\item If $(q_{n})_{n=1}^{\infty}$ is any sequence of integers, then
\[\sigma_{n}^{\oplus q_{n}}\colon G\to \Sym(\{1,\cdots,d_{n}\}\times \{1,\cdots, q_{n}\})\]
 given by $\sigma_{n}^{\oplus q_{n}}(g)(j,r)=(\sigma_{n}(g)(j),r)$ are approximate homomorphisms with stabilizer type $\Theta$.
\item If $(r_{n})_{n}$ is any sequence of integers so that $\frac{r_{n}}{d_{n}}\to 0$, define $\sigma_{n}\oplus t_{r_{n}}\colon G\to \Sym(d_{n}+r_{n})$ by \[(\sigma_{n}\oplus t_{r_{n}})(g)(j)=\begin{cases}
\sigma_{n}(g)(j), \textnormal{ if $1\leq j\leq d_{n}$}\\
j,\textnormal{ if $d_{n}+1\leq j\leq d_{n}+r_{n}$}
\end{cases},\]
then $\sigma_{n}\oplus t_{r_{n}}$ are approximate homomorphisms with stabilizer type $\Theta$.
\item \label{item: speed change} If $s\colon \bN\to \bN$ is any function so that
\[\lim_{n\to\infty}s(n)=+\infty,\]
then $\sigma_{s(n)}$ are approximate homomorphisms with stabilizer type $\Theta$.

\end{enumerate}

\end{prop}

\begin{proof}
These are all exercises.
\end{proof}

We also need the following analogue of \cite[Proposition 6.1]{AP1}, which is proved exactly as in \cite[Lemma 7.6]{BLT} using Proposition \ref{prop: permanence}.

\begin{prop}\label{prop:change the sequence}
Let $G$ be a  countable discrete group and $\Theta\in \IRS(G)$. Suppose there is a sequence of integers $k_{n}$ with $k_{n}\to\infty$ and approximate homomorphisms $\psi_{n}\colon G\to \Sym(k_{n})$ with stabilizer type $\Theta$.  Then for any sequence of integers $d_{n}\to \infty$, there are approximate homomorphisms $\sigma_{n}\colon G\to \Sym(d_{n})$ with stabilizer type $\Theta$.
\end{prop}



We now have the tools to explain why Theorem \ref{thm:equivalence of IRS means conjugacy} implies \cite[Theorem 1.3]{BLT}.

\begin{cor}[Theorem 1.3 of \cite{BLT}] \label{cor:BLT}
Let $G$ be an amenable group. Then $G$ is permutation stable if and only if for every $\Theta\in \IRS(G)$, there is a sequence of positive integers $(d_{n})_{n}$ and a sequence of homomorphism $\kappa_{n}\colon G\to \Sym(d_{n})$ with $(\Stab_{\kappa_{n}})_{*}(u_{d_{n}})\to_{n\to\infty}\Theta$ in the weak$^{*}$-topology. Here $\Stab_{\kappa_{n}}$ is the map $j\mapsto \{g\in G:\kappa_{n}(g)(j)=j\}$.
\end{cor}

\begin{proof}
First suppose that for every $\Theta\in \IRS(G)$, there is a sequence of positive integers $(d_{n})_{n}$ and a sequence of homomorphisms $\kappa_{n}\colon G\to \Sym(d_{n})$ with $(\Stab_{\kappa_{n}})_{*}(u_{d_{n}})\to_{n\to\infty}\Theta$ in the weak$^{*}$-topology. Let $\sigma_{n}\colon G\to \Sym(k_{n})$ be asymptotic homomorphisms, and fix $\omega\in \beta\bN\setminus \bN$. Let $\Theta=\IRS(\sigma_{\omega})$. By assumption, we can find  a sequence of homomorphisms $\kappa_{n}\colon G\to \Sym(d_{n})$ with $(\Stab_{\kappa_{n}})_{*}(u_{d_{n}})\to_{n\to\infty}\Theta$ in the weak$^{*}$-topology. By Proposition \ref{prop:change the sequence}, we may assume that $d_{n}=k_{n}$. Set $\kappa_{\omega}=(\kappa_{n})_{n\to\omega}$. Then by Theorem \ref{thm:equivalence of IRS means conjugacy}, $\sigma_{\omega},\kappa_{\omega}$ are conjugate. Since $\omega$ was arbitrary, this implies that there is a sequence of permutations $\chi_{n}\in \Sym(d_{n})$ with
\[d_{\Hamm}(\sigma_{n}(g),\chi_{n}\kappa_{n}(g)\chi_{n}^{-1})\to_{n\to\infty}0 \textnormal{ for all $g\in G$}.\]
Since $g\mapsto \chi_{n}\kappa_{n}(g)\chi_{n}^{-1}$ are homomorphisms, this proves that $G$ is permutation stable.

Now suppose that $G$ is permutation stable. Let $\Theta\in \IRS(G)$, and choose an action $G\actson (X,\nu)$ with $\Stab_{*}(\nu)=\Theta$ (e.g. the $\Theta$-Bernoulli action of $G$ with base $([0,1],m)$). Let $\cR$ be the orbit equivalence relation of $G\actson (X,\nu)$. Since $G$ is amenable, \cite{OWAnnounce}, \cite[II \S 3]{OrnWeiss} (see also  \cite{CFW}) imply that $\cR$ is hyperfinite and thus sofic (see e.g. \cite[Proposition 3.4]{PaunSofic}). By Proposition \ref{prop: equivalent definition of a sofic rep}, this implies that for every free ultrafilter $\omega$, we may find a trace-preserving homomorphism
\[\pi\colon L(\cR)\to \cM:=\prod_{n\to\omega}(M_{d_{n}}(\bC),\tr)\]
such that $\pi(L^{\infty}(X))\subseteq L^{\infty}(\cL)$ and $\pi([\cR])\subseteq \cS$. For $g\in G$, let $p_{g}=1_{\Fix(g)}\in L^{\infty}(X)$ and apply \cite[Lemma 5.4.2 (i)]{AP} to find $B_{g,n}\subseteq \{1,\cdots,d_{n}\}$ with $\pi(p_{g})=(1_{B_{g,n}})_{n\to\omega}$. For $g\in G$, let $\pi(\lambda(g))=(\sigma_{n}(g))_{n\to\omega}$. We first note the following.

\emph{Claim: $u_{d_{n}}(B_{g,n}\Delta \Fix(\sigma_{n}(g)))\to_{n\to\omega}0$ for every $g\in G$.}
To prove the claim, note that:
\[\lambda(g)p_{g}=p_{g},\textnormal{ and }\tau(\lambda(g)(1-p_{g}))=0.\]
Let $p_{g,n}=1_{B_{g,n}}\in \ell^{\infty}(d_{n})$ which we view as a subset of $M_{n}(\bC).$ Thus:
\[2u_{d_{n}}(\{j\in B_{g,n}:\sigma_{n}(g)(j)\ne j\})=\|\sigma_{n}(g)p_{g,n}-p_{g,n}\|_{2}^{2}\to_{n\to\omega}0,\]
the last part following as $\sigma_{n}$ is a sofic approximation and $\lambda(g)p_{g}=p_{g}$.
On the other hand,
\[u_{d_{n}}(\{j\in B_{g,n}^{c}:\sigma_{n}(g)(j)= j\})=\tr(\sigma_{n}(g)(1-p_{g,n}))\to_{n\to\omega}\tau(\lambda(g)(1-p_{g}))=0,\]
the second-to-last part following as $\sigma_{n}$ is a sofic approximation. This proves the claim.

For $F\subseteq G$ let $\mathcal{E}_{F}\in C(\Sub(G))$ be the indicator function of $\{H:F\subseteq H\}$.
The claim similarly implies that for $F\subseteq G$ is finite if $\pi(\cE_{F})=(1_{B_{F,n}})_{n\to\omega}$, then
\[u_{d_{n}}\left(B_{F,n}\Delta\bigcap_{g\in F}\Fix(\sigma_{n}(g))\right)\to_{n\to\omega}0.\]
Setting $\varsigma_{F,n}=\bigcap_{g\in F}\Fix(\sigma_{n}(g))$, we then have $\pi(\cE_{F})=(1_{\varsigma_{F,n}})_{n\to\omega}$. Since $G$ is permutation stable, we may find honest homomorphisms $\kappa_{n}\colon G\to \Sym(n)$ so that $\sigma\circ \Xi=(\kappa_{n})_{n\to\omega}$ where $\Xi\colon G\to [\cR]$ is the map $\Xi(g)(x)=gx$. So for any finite $F\subseteq G$:
\begin{align*}
  \lim_{n\to\omega}(\Stab_{\kappa_{n}})_{*}(\{H:F\subseteq H\})&=\lim_{n\to\omega}u_{d_{n}}(\{j:\kappa_{n}(g)(j)=j \textnormal{ for all $g\in F$}\})\\
  &=\lim_{n\to\omega}u_{d_{n}}(\{j:\sigma_{n}(g)(j)=j \textnormal{ for all $g\in F$}\})\\
  &=\lim_{n\to\omega}u_{d_{n}}(\varsigma_{F,n})\\
  &=\lim_{n\to\omega}u_{d_{n}}(B_{F,n})\\
  &=\tau(1_{\cE_{F}})\\
  &=\nu(\{x:gx=x  \textnormal{ for all $g\in F$\}})\\
  &=\Theta(\{H:F\subseteq H\}).
\end{align*}
As in Corollary \ref{cor:AP stability} this implies that
$\lim_{n\to\omega}(\Stab_{\kappa_{n}})_{*}(u_{n})=\Theta.$ Since this holds for all $\omega$, we may find a sequence of honest homomorphisms  $\kappa_{n}\colon G\to \Sym(n)$ with
\[\lim_{n\to\infty}(\Stab_{\kappa_{n}})_{*}(u_{n})=\Theta.\]
\end{proof}

The proof of the reverse implication given above is similar to the one in \cite[Proposition 3.15]{APStable}.

\subsection*{Acknowledgments}
The authors thank Andreas Thom and Liviu P\u{a}unescu for helpful comments on an earlier version of this paper.
B. Hayes gratefully acknowledges support from the NSF grant DMS-2144739. S. Kunnawalkam Elayavalli gratefully acknowledges support from the Simons Postdoctoral Fellowship. Both authors thank the anonymous referee for their numerous comments, which greatly improved the paper.

%
%

\end{document}